\documentclass[a4paper, 10pt]{amsart}
\usepackage{verbatim}

\newcommand{\cal}{\mathcal}

\newcommand{\R}{{\mathbb R}}
\newcommand{\N}{{\mathbb N}}

\newcommand{\Ha}{{\cal H}}

\newcommand{\eps}{\varepsilon}
\newcommand{\sL}{\cal{L}}
\newcommand{\sS}{\cal{S}}
\newcommand{\sM}{\cal{M}}
\newcommand{\usS}{\overline{\cal S}}
\newcommand{\usM}{\overline{\cal M}}
\newcommand{\lsS}{\underline{\cal S}}
\newcommand{\lsM}{\underline{\cal M}}
\newcommand{\ldim}{\underline{\dim}}
\newcommand{\udim}{\overline{\dim}}
\newcommand{\dbox}{[0,1]^{d-1}}
\newcommand{\bd}{\partial}

\newtheorem{thm}{Theorem}
\newtheorem{prop}[thm]{Proposition}
\newtheorem{lem}[thm]{Lemma}
\newtheorem{cor}[thm]{Corollary}

\newtheorem{rem}[thm]{Remark}
\numberwithin{equation}{section} \numberwithin{thm}{section}

\begin{document}
\title{Lower S-dimension of fractal sets}
\author{Steffen Winter}
\address{Karlsruhe Institute of Technology, Department of Mathematics, 76128 Karlsruhe, Germany}
\thanks{The author was supported by a cooperation grant of the Czech and the German science foundation, DFG project no.\ WE 1613/2-1.} 
\date{\today}

\begin{abstract}
 The interrelations between (upper and lower) Minkowski contents  and (upper and lower) surface area based contents (S-contents) as well as between their associated dimensions have recently been investigated for general sets in $\R^d$ (cf.\ \cite{rw09}). While the upper dimensions always coincide and the upper contents are bounded by each other, the bounds obtained in \cite{rw09} suggest that there is much more flexibility for the lower contents and dimensions.   

We show that this is indeed the case. There are sets whose lower S-dimension is strictly smaller than their lower Minkowski dimension. More precisely, given two numbers $s,m$ with $0<s<m<1$, we construct sets $F$ in $\R^d$ with lower S-dimension $s+d-1$ and lower Minkowski dimension $m+d-1$. In particular, these sets are used to demonstrate that the inequalities obtained in \cite{rw09} regarding the general relation of these two dimensions are best possible. 
\end{abstract}
\maketitle

\section{Introduction}
For a bounded set $A\subset\R^d$ and $r\ge 0$, let 
$$
A_r:=\{x\in\R^d: \inf_{a\in A}|x-a|\le r\}
$$
be the \emph{$r$-parallel set} (or \emph{$r$-neighbourhood}) of $A$. Write $V(A_r):=\lambda_d(A_r)$ for the volume 
of $A_r$ and $\Ha^{d-1}(\bd A_r)$ for the surface area of its boundary. ($\lambda_d$ is the Lebesgue measure and $\Ha^t$ the $t$-dimensional Hausdorff measure.) 
Recall that the \emph{$s$-dimensional lower and upper Minkowski content} of $A$ are defined by
\[
\underline{\cal M}^s(A):=\liminf_{r\to 0} \frac{V(A_r)}{\kappa_{d-s}r^{d-s}} \quad \text{ and } \quad 
\overline{\cal M}^s(A):=\limsup_{r\to 0} \frac{V(A_r)}{\kappa_{d-s}r^{d-s}},
\] 
where 
$\kappa_t:=\pi^{t/2}/\Gamma(1+\frac t2)$. For integers $t$, $\kappa_t$ is the volume of the unit ball in $\R^t$.
If $\underline{\cal M}^s(A)=\overline{\cal M}^s(A)$, then the common value ${\cal M}^s(A)$ is the \emph{$s$-dimensional Minkowski content} of $A$. 
Denote by 
\begin{align*}
\underline{\dim}_M A:=\inf\{t\ge 0 : \underline{\cal M}^s(A)=0\}
&&\text{ and }
&&\overline{\dim}_M A:=\inf\{t\ge 0 :\overline{\cal M}^s(A)=0\}
\end{align*}
the \emph{lower} and \emph{upper Minkowski dimension} of $A$. If both numbers coincide, the common value $\dim_M A$ is the \emph{Minkowski dimension} of $A$. It is well known that the Minkowski dimension coincides with the box counting dimension, cf.\ for instance \cite{fal1} or \cite{mattila}. 
See also the beginning of Section~\ref{proof2} for alternative definitions of $\dim_M$. 

Minkowski contents and Minkowski dimension have many applications, for instance in the theory of fractal strings and sprays, where the spectral properties of a domain have been shown to be deeply connected with 
the Minkowski content of its boundary, see \cite{FGCD} and the references therein; and in the study of singular integrals, cf.~\cite{zubrinic}. Box counting methods are widely used in the applied sciences to estimate the fractal dimension, i.e. $\dim_M$, of 'rough' objects, cf.~\cite{fal1}. Some variant of the Minkowski content has been proposed as a texture parameter (\emph{lacunarity}) for finer classifications, cf.~\cite{mandelbrot}. It seems therefore of vital interest to illuminate further the geometric meaning and the mathematical properties of Minkowski contents, for instance by providing alternative definitions and studying related concepts. 

One of these is the notion of \emph{S-content} (or \emph{surface area based content}), arising when in the definition of the Minkowski content the volume $V(A_r)$  is replaced with the surface area $\Ha^{d-1}(\bd A_r)$. It was studied in \cite{rw09}. 
For $0\leq s<d$, let
\[
\underline{\cal S}^s(A):=\liminf_{r\to 0} \frac{\Ha^{d-1}(\bd A_r)}{(d-s)\kappa_{d-s}r^{d-1-s}} 
\quad \text{ and } \quad 
\overline{\cal S}^s(A):=\limsup_{r\to 0} \frac{\Ha^{d-1}(\bd A_r)}{(d-s)\kappa_{d-s}r^{d-1-s}}
\]  
denote the \emph{lower} and \emph{upper $s$-dimensional S-content} of $A$. If both numbers coincide,  
the common value ${\cal S}^s(A)$ is the \emph{($s$-dimensional) S-content} of $A$. For convenience, we set ${\cal S}^d(A):=0$ (which is well motivated by the fact that $\lim_{r\to 0} r \Ha^{d-1}(\bd A_r)=0$, cf.~\cite[p.4]{rw09}). The numbers 
\begin{align*}
\underline{\dim}_S A:=
\inf\{t \ge 0: \underline{\cal S}^t(A)=0\}
\quad \text{ and } \quad
\overline{\dim}_S A:=
\inf\{t\ge 0: \overline{\cal S}^t(A)=0\}
\end{align*}
are the \emph{lower and upper S-dimension} of $A$, respectively, and, if they coincide, the common value $\dim_S A$ will be called \emph{S-dimension} of the set $A$.  

The S-content is not only a natural counterpart to the Minkowski content. Both contents appear as special cases in the framework of fractal curvatures.  More precisely, Minkowski content and S-content are (up to normalization) the  fractal curvatures of order $d$ and $d-1$, whenever the respective limits exist. Fractal curvature measures have been introduced as a generalization of curvature measures to very singular sets by means of approximation with parallel sets. The fractal curvatures are the total masses of these measures. 
They form a set of $d+1$ parameters characterizing the geometry of fractal sets beyond dimension, see \cite{winter, survey, zaehle} for definitions and more details. 

Based on the fundamental observation that the boundary surface area of $A_r$ is the derivative of its volume, cf.\ Stacho \cite{stacho}, it has been investigated in \cite{rw09} under which assumptions Minkowski content and S-content coincide. 
In particular, the following results have been obtained regarding the general relation between Minkowski contents and S-contents.
\begin{thm} \cite[Cor.~3.2 and 3.6]{rw09} \label{thmU}\\
Let $A\subset\R^d$ be a compact set with $V(A)=0$. Then, for $0\le s\le d$,
\begin{align}  \label{thmU-eq}
\frac{d-s}d \usS^s(A)\le \usM^s(A) \le \usS^s(A).
\end{align}
Consequently, $\udim_S A= \udim_M A$.
\end{thm}
Note that the left inequality in \eqref{thmU-eq} remains valid for sets $A$ with $V(A)>0$, while the right inequality may fail in this case and the upper S-dimension may be strictly smaller than the upper Minkowki dimension. 
The inequalities obtained in \cite{rw09} for the lower contents and dimensions are much weaker:
\begin{thm} \cite[ Cor.~3.2 and Prop.~3.7]{rw09} \label{thmL}\\
Let $A\subset\R^d$ be a compact set with $V(A)=0$. Then, for $0\le s\le d$,
\begin{align}\label{thmL-eq1}
c \left(\lsM^{s\frac d {d-1}}(A)\right)^{\frac{d-1}d} \le \lsS^s(A)\le \lsM^s(A),
\end{align}
where $c$ is an (explicitely known) constant depending only on $d$ and $s$. Consequently,
\begin{align}\label{thmL-eq2}
\frac{d-1}{d} \ldim_M A\le \ldim_S A\le \ldim_M A.
\end{align}
\end{thm}

Combining the above theorems, it follows immediately, that the existence of the S-content implies the existence of the Minkowski content and both notions coincide (for sets in $\R^d$ with $V(A)=0$). If lower and upper S-content differ, the situation is more delicate.
In \cite[cf.\ Example 3.3]{rw09}, the Sierpinski gasket has been discussed, which shows that the lower S-content can be strictly smaller than the lower Minkowski content. The lower dimensions coincide in this case, in fact, the dimensions exist and coincide. 
However, the inequalities in \eqref{thmL-eq2} suggest that either they can be improved (to equality for the lower dimensions) or there are sets whose lower S-dimension is strictly smaller than their lower Minkowski dimension. This was one of the most pressing questions left open in \cite[cf.\ the second Remark on p.10]{rw09}. 

In this note we show that for any $d\in\N$ there exist sets $A\subset\R^d$ with  $\ldim_S A< \ldim_M A$ and, moreover, that the lower S-dimension can assume any value between the upper and the lower bound given in \eqref{thmL-eq2}, showing, in particular, that these bounds are optimal. The essential construction is done for $d=1$ using the concept of fractal strings, which goes back to \cite{LapPo1}, see also the monograph \cite{FGCD}. The result in higher dimensions is based on a Cartesian product argument.
The paper is organized as follows. In the next section, the sets are constructed and the main results are stated. The proof for $d=1$ is discussed in Section \ref{proof1} and for $d\ge2$ in Section \ref{proof2}, where also some more general statements regarding the S-dimension of product sets are derived.

\section{Main results} \label{results}
Let two numbers $s,m$ be given with $0<s<m<1$. Set $q:=1+\frac{1}{s}-\frac 1m$. Let ${\sL}=\sL(s,m)=(l_j)_{j=1}^\infty$ be the \emph{fractal string} (i.e., a nonincreasing sequence of nonnegative real numbers; cf.~\cite[p.1]{FGCD}) containing $[2^{q^{k+1}\cdot s}]$ times the ``length'' $2^{-q^k}$, $k=1,2,\ldots$, where $[x]$ denotes the integer part of a number $x\in\R$.  
Observe that 
\[
L:=\sum_{j=1}^\infty l_j = \sum_{k=1}^\infty [2^{q^{k+1}\cdot s}]\cdot 2^{-q^k}\le \sum_{k=1}^\infty 2^{q^{k+1}\cdot s}\cdot 2^{-q^k}=\sum_{k=1}^\infty 2^{q^{k}(q\cdot s-1)}<\infty,
\]
since $q\cdot s=1+s-\frac sm<1$. Hence $\sL$ has a geometric realization as a union of disjoint open intervals $I_j$ of lengths $l_j$ in $\R$ such that the total length $\lambda_1(\Omega)$ of $\Omega:=\bigcup_{j=1}^\infty I_j$ is finite. For simplicity, we assume that the $I_j$ are all subsets of some open interval $I$ of length $L$. (Note that the term \emph{fractal string} is also frequently used for the set $\Omega$, cf.\ e.g.~\cite[p.9]{FGCD}.)
  
Let $F=F(s,m)$ denote the boundary of (an arbitrary but fixed) geometric realization $\Omega$ of $\sL$ in $I$, i.e., $F=\partial \Omega$. Note that the latter assumption implies $\overline{I}=\Omega\cup F$ and $\lambda_1(F)=0$.

\begin{thm}\label{thm:F}
For $0<s<m<1$, the set $F=F(s,m)\subset\R$ has lower S-dimension $\ldim_S F= s$ and lower Minkowski dimension $\ldim_M F=m$. Moreover, the upper Minkowski and S-dimension of $F$ are given by 
$$
\udim_M F=\udim_S F = s\cdot q = 1+s-\frac sm.
$$  
\end{thm}

For $d=1,2,\ldots$, let $F_d=F_d(s,m):=F(s,m)\times [0,1]^{d-1}\subset\R^d$ be the Cartesian product of the set $F$ and the $(d-1)$-dimensional unit cube $[0,1]^{d-1}$.  
\begin{thm} \label{thm:Fd}
For $0<s<m<1$ and $d\in\N$, the set $F_d=F_d(s,m)\subset\R^d$ has lower S-dimension $\ldim_S F_d = s + d -1$ and lower Minkowski dimension $\ldim_M F_d = m+ d-1$.
The upper Minkowski and S-dimension of $F_d$ are given by 
$$
\udim_M F_d=\udim_S F_d = s\cdot q+d-1 = d+s-\frac sm.
$$   
\end{thm}

The proofs of Theorems~\ref{thm:F} and \ref{thm:Fd} are given in Sections~\ref{proof1} and \ref{proof2}, respectively. In the course of the proof of Theorem~\ref{thm:F} we will also derive the precise expressions for the upper and lower contents of the sets $F(s,m)$. The proof of Theorem~\ref{thm:Fd}
is based on some more general statements on the Minkowski and S-dimension of product sets.
 
Now recall from \eqref{thmL-eq2} that, for arbitrary compact sets $A\subset \R^d$, we have
\begin{align*}
\frac{d-1}{d} \ldim_M A\le \ldim_S A\le \ldim_M A.
\end{align*}
The above results clearly show that the lower S-dimension can be strictly smaller than the lower Minkowski dimension, i.e., the right hand side inequality can be strict. This is in sharp contrast to the situation for the upper dimensions, which do always coincide.  
Moreover, the above Theorems show that the constant $\frac{d-1}d$ for the lower bound is optimal: 

\begin{cor}
For any $d\in\N$ and any constant $c$ such that $\frac{d-1}d<c\le 1$ there exists a set $A\subset\R^d$ such that $c\cdot \ldim_M A= \ldim_S A$.
\end{cor}
\begin{proof}
The case $c=1$ is not covered by the class of sets above, however, examples of such sets are known. For instance, if $F$ is any non-arithmetic self-similar set in $\R^d$ satisfying the open set condition and with similarity dimension $D<d$, then, by \cite[Theorem 4.5]{rw09}, 
$\dim_S F=\dim_M F=D$.
 
Fix $d\in\N$ and $c$ such that $\frac{d-1}d<c<1$. Set $s:=c-\frac{d-1}d$ and $m:=\frac 1c((1-c)(d-1) +s)$. Then $0<s<m<1$ (since $m>mc=(1-c)(d-1)+s>s$ and $mc=(1-c)(d-1)+s<\frac{d-1}d +s= \frac{d-1}d + c -\frac{d-1}d=c$) and so, by Theorem~\ref{thm:Fd},
the set $A:=F_d(s,m)$ has $\ldim_S A=s+d-1$ and $\ldim_M A = m+d-1$. Hence
\[
c\cdot \ldim_M A= c(d-1+m)=c(d-1)+(1-c)(d-1)+s=d-1+s=\ldim_S A,
\]
i.e., the set $A$ satisfies the desired equality.
\end{proof}
\begin{rem}
The class of sets discussed does not provide examples for the case $c=\frac{d-1}d$, i.e., sets $A$ for which the lower bound in \eqref{thmL-eq2} is sharp. Thus the following question remains open: Does there exist a set $A\subset\R^d$ for which $\ldim_S A =\frac{d-1} d \ldim_M A$?  Another open question is, whether $\ldim_S A=\ldim_M A$ implies $\ldim_M A=\udim_M A$ or vice versa, i.e, whether the equivalence of the lower dimensions is related to the existence of the Minkowski dimension in some way. The examples considered so far suggest such a relation, at least they do not disprove it.
\end{rem}

We notice that it is also possible to prescribe lower and upper S-dimension and find a set with these S-dimensions within the class of sets discussed.
\begin{cor}\label{cor2}
Let $0<s<u<1$. There exists a set $A\subset\R^d$ such that $\ldim_S A=s+d-1$ and $\udim_S A=u+d-1$. 
\end{cor} 
\begin{proof} Set  $m:=\frac s{1+s-u}$ and note that $s<m<1$.  Let $A:=F_d(s,m)$. We have $q=1+\frac 1s -\frac 1m =1+\frac 1s -\frac{1+s-u}{s} =\frac u s$. Hence, by Theorem~\ref{thm:Fd}, 
$\ldim_S A=s+d-1$ and $\udim_S=qs+d-1=u+d-1$.
\end{proof}
Corollary~\ref{cor2} shows that the difference between the upper and the lower S-dimension of a set in $\R^d$ may be any number between $0$ and $1$. 
For $d=1$ this implies that the trivial lower bound $0=0\cdot\udim_S A\le\ldim_S A$  for $\ldim_S$ in terms of $\udim_S$ is the best possible for general compact sets in $\R$. However, this is also an immediate consequence of  the well known fact that there exist sets $A$ in $\R$ with $\ldim_M A=0$ and $\udim_M A=1$ (taking into account Theorems~\ref{thmU} and \ref{thmL}). Hence there is no general restriction on the difference between upper and lower S-dimension for sets in $\R$ apart from the trivial ones. It remains open whether this difference can be larger for sets in $\R^d$, $d\ge 2$. 

For completeness, we remark that similarly as in Corollary~\ref{cor2} one can also prescribe $\ldim_M$ and $\udim_M$ within $(d-1,d)$ and find a set in $\R^d$ (within the class of sets discussed) with these Minkowski dimensions. 
\begin{cor}
Let $0<m<u<1$. There exists a set $A\subset\R^d$ such that $\ldim_M A=m+d-1$ and $\udim_M A=u+d-1$. 
\end{cor} 
We leave the simple proof as an exercise, also because results of this type are known, cf.\ for instance \cite[Section 5.3, p.77]{mattila} and \cite{zubrinic}. A better result is obtained in \cite[Theorem 1.2]{zubrinic}, which is in fact optimal: It is possible to prescribe numbers $\underline{d}\le \overline{d}$ in $[0,d]$ and find a set $A\subset\R^d$ such that $\ldim_M A=\underline{d}$ and $\udim_M A=\overline{d}$.

We note that fractal strings of a similar type as the ones used here to construct the sets $F(s,m)$ appear in \cite[cf.\ Examples 3.12-3.14]{LapPo1}, where they are used to demonstrate that certain implications in connection with one-sided (lower) estimates generalizing the modified Weyl-Berry conjecture are nonreversible, in general; see \cite[Theorem 3.11]{LapPo1} for more details. It is an interesting question whether (lower) S-contents play a role in this context.

\section{Proof of Theorem~\ref{thm:F}} \label{proof1}
For a fractal string $\sL=(l_j)_{j=1}^\infty$, let $(r_k)_{k=1}^\infty$ be the (ordered) sequence of the lengths occuring in $\sL$, i.e.,  $r_1>r_2>r_3>\ldots>0$ and $\{l_j:j\in\N\}=\{r_k:k\in\N\}$. 
For $k=1,2,\ldots$, let  
$$
N_k:=\#\{j\ge 1: l_j=r_k\},
$$
denote the multiplicity of the $k$-th length $r_k$ in $\sL$. 
For convenience, we set $N_0:=1$ and $r_0:=\infty$.

Let $0<s<m<1$ and let $F=F(s,m)$ as defined in Section~\ref{results}. Recall that $q=1+\frac{1}{s}-\frac 1m$.
For the fractal string $\sL=(l_j)_{j=1}^\infty$ associated with $F$ we have
$N_k=[2^{q^{k+1}\cdot s}]$ and $r_k=2^{-q^k}$,  $k=1,2,\ldots$.
For the computation of the upper and lower S-content of $F(s,m)$ we require the following simple fact. 
\begin{lem}\label{lem:summe}
Let $a,b>1$ and $\eps>0$. There exists a number $k_0=k_0(a,b,\eps)$ such that for $k\ge k_0$
$$
\sum_{i=1}^{k} a^{b^{i}}\le (1+\eps) a^{b^k}.
$$
\end{lem}
\begin{proof}
Since $a^{b^k(1-b)}\cdot k\to 0$ as $k\to\infty$, it is possible to choose $k_0$ such that
$$
a^{b^{k_0}(1-b)}\cdot (k_0)<\eps.
$$
If necessary, enlarge $k_0$ such that the sequence $(a^{b^k(1-b)}\cdot k)_{k\ge k_0}$ is monotone decreasing.
Then
$$
a^{b^{k-1}}<a^{b^k}\cdot \frac\eps{k-1} \quad \mbox{ for } k\ge k_0, 
$$
and, since $(a^{b^i})_{i\in\N}$ is monotone increasing, 
$$
a^{b^i}<a^{b^k}\cdot \frac\eps{k-1} \quad \mbox{ for } k\ge k_0, i=1,\ldots, k-1.
$$
Now the assertion follows by summing up over $i=1,\ldots, k$.
\end{proof}
 
\begin{prop}\label{umink} 
For $F=F(s,m)$,
$$
\overline{\sS}^{s\cdot q}(F)=(1-sq)^{-1}\kappa_{1-sq}^{-1}2^{1-s\cdot q}.
$$
Hence, in particular, $\udim_M F =\udim_S F=s\cdot q$.
\end{prop}
\begin{proof}
Let $t>0$. For $2r\in[r_k,r_{k-1})$, $k=1,2,\ldots$, we have
$$ 
r^t \Ha^0(\bd F_r)= r^t 2\sum_{i=0}^{k-1} N_i\le \left(\frac{r_{k-1}}2\right)^t 2\sum_{i=0}^{k-1} N_i,
$$
since the function $f(x)=x^t$ is monotone increasing. Hence
\begin{align} \label{uS1}
(1-t)\kappa_{1-t} \overline{\sS}^t(F)&=\limsup_{r\to 0} r^t  2\sum_{i=0}^{k-1} N_i = \limsup_{k\to\infty} 2^{1-t} r_{k-1}^t  \sum_{i=0}^{k-1} N_i\,.
\end{align}
Since
$ 
N_i=[2^{q^{i+1}\cdot s}]\le 2^{q^{i+1}\cdot s}
$, 
for $i=1,2,\ldots$, and $N_0=1<2^{q^1\cdot s}$ we have
\begin{align} \label{uS2}
2^{q^{k}\cdot s}\le 1+N_{k-1}\le \sum_{i=0}^{k-1} N_i\le \sum_{i=0}^{k-1} 2^{q^{i+1}\cdot s}.
\end{align}
Applying Lemma~\ref{lem:summe} with $a=2^s>1$ and $b=q>1$, we infer that for each $\eps>0$ there exists a $k_0=k_0(\eps)$ such that
\begin{align} \label{uS3}
\sum_{i=0}^{k-1} 2^{q^{i+1}\cdot s}\le (1+\eps)\cdot 2^{q^k \cdot s}, 
\end{align}
for each $k\ge k_0$. 
Thus, on the one hand,
\begin{align*}
(1-t)\kappa_{1-t} \overline{\sS}^t(F)&\ge  \limsup_{k\to\infty} 2^{1-t} 2^{-q^{k-1}\cdot t}  2^{q^{k}\cdot s}=2^{1-t} \lim_{k\to\infty} 2^{q^{k-1}(qs-t)}\,,
\end{align*}
and on the other hand
\begin{align*}
(1-t)\kappa_{1-t}  \overline{\sS}^t(F)&\le  \limsup_{k\to\infty} 2^{1-t} 2^{-q^{k-1}\cdot t} (1+\eps) 2^{q^{k}\cdot s}=2^{1-t}(1+\eps) \lim_{k\to\infty} 2^{q^{k-1}(qs-t)}\,.
\end{align*}
Since the latter holds for each $\eps>0$, we conclude 
$$
\overline{\sS}^t(F)= \left\{
\begin{array}{ll}
0& \text{ if } t>sq\\
(1-sq)^{-1}\kappa_{1-sq}^{-1} 2^{1-sq} & \text{ if } t=sq\\
\infty & \text{ if } t\le sq
\end{array}
\right.\,.
$$
Since the upper dimensions coincide, cf.~Theorem~\ref{thmU}, this implies in particular $\udim_M F=\udim_S F=s\cdot q$.
\end{proof}

\begin{rem} Theorem~\ref{thmU} implies that 
$$
(1-sq) \overline{\sS}^{sq}(F)\le\overline{\sM}^{sq}(F)\le\overline{\sS}^{sq}(F).
$$
With slightly more effort one can show that, in fact,  $\overline{\sM}^{sq}(F)=\overline{\sS}^{sq}(F)$ holds. 
\end{rem}

A similar argument allows to compute the lower S-content of $F$.  
\begin{prop}\label{lSdim}
For $F=F(s,m)$,
$$
\underline{\sS}^{s}(F)=(1-s)^{-1}\kappa_{1-s}^{-1}2^{1-s}.
$$
Hence, in particular, $\ldim_S F=s$.
\end{prop}
\begin{proof}
Let $t>0$. A similar argument as for \eqref{uS1} shows that
\begin{align*} 
(1-t)\kappa_{1-t} \underline{\sS}^t(F)&=\liminf_{r\to 0} r^t  2\sum_{i=0}^{k-1} N_i = \liminf_{k\to\infty} 2^{1-t} r_{k}^t  \sum_{i=0}^{k-1} N_i\,.
\end{align*}
Taking into account \eqref{uS2} and \eqref{uS3}, we infer that on the one hand
\begin{align*} 
(1-t)\kappa_{1-t}\underline{\sS}^t(F)&\ge \liminf_{k\to\infty} 2^{1-t} 2^{-q^{k}\cdot t} 2^{q^{k}\cdot s}= 2^{1-t} \lim_{k\to\infty} 2^{q^{k}(s-t)}\,,
\end{align*}
and on the other hand, for each $\eps>0$,
\begin{align*}
(1-t)\kappa_{1-t}  \underline{\sS}^t(F)&\le  \liminf_{k\to\infty} 2^{1-t} 2^{-q^{k}\cdot t} (1+\eps) 2^{q^{k}\cdot s}=2^{1-t}(1+\eps) \lim_{k\to\infty} 2^{q^{k}(s-t)}\,.
\end{align*}
This implies $\underline{\sS}^s(F)=(1-s)^{-1}\kappa_{1-s}^{-1}2^{1-s}$ and $\ldim_S F=s$ as asserted.
\end{proof}

The computation of the lower Minkowski content is more involved. We will employ the following two simple statements.
\begin{lem}\label{lem:minimum}
For $L,M>0$ and $0<D<1$, the function $h=h_{M,L,D}:(0,\infty)\to\R$, defined by
$$
h(x)=x^D M + x^{D-1} L,
$$
has its global minimum at $x_{\min}=x_{\min}(M,L,D):=\frac{(1-D) L}{DM}$. Moreover, 
$$
h(x_{\min})=\left(\frac{(1-D)^D}{D^D}+\frac{(1-D)^{D-1}}{D^{D-1}}\right) L^D M^{1-D}=D^{-D}(1-D)^{D-1} L^D M^{1-D}.
$$
\end{lem}

\begin{lem}\label{lem:reihe}
Let $a,b>1$ and $\eps>0$. There exists a number $k_0=k_0(a,b,\eps)$ such that for $k\ge k_0$
$$
\sum_{i=k}^\infty a^{-b^{i}}\le (1+\eps) a^{-b^k}.
$$
\end{lem}
 
\begin{prop}\label{lMdim}
For $F=F(s,m)$,
$$
\underline{\sM}^{m}(F)=\kappa_{1-m}^{-1} m^{-m} (1-m)^{m-1}.
$$
Hence, in particular, $\ldim_M F=m$.
\end{prop}

\begin{proof}
Let $0<t<1$. For $2r\in[r_k,r_{k-1})$, $k=1,2,\ldots$, we have
$$ 
r^{t-1} \lambda_1(F_r)= r^{t} 2\sum_{i=0}^{k-1} N_i + r^{t-1} \sum_{i=k}^{\infty} N_i r_i.
$$
Setting $M_k:=2\sum_{i=0}^{k-1} N_i$ and $L_k:=\sum_{i=k}^{\infty} N_i r_i$, we infer from Lemma~\ref{lem:minimum}, that the global minimum of the function $h_{M_k, L_k, t}(x)=x^t M_k + x^{t-1} L_k$ is 
$$
x_k=\frac{1-t}{t} \frac {L_k}{M_k}=\frac{1-t}{t} \frac{\sum_{i=k}^{\infty} N_i r_i}{2\sum_{i=0}^{k-1} N_i}.
$$
We claim that there exists a number $k'\in\N$ such that, for all $k\ge k'$,
\begin{align}\label{eq:xk-in}
r_k< 2 x_k< r_{k-1},
\end{align}
i.e., the global minimum of $h_{M_k,L_k, t}$ is contained in the interval $(r_k/2,r_{k-1}/2)$. 

For a proof of \eqref{eq:xk-in}, fix some $\eps>0$. Observe that there exists $k_0\in\N$ such that 
\begin{align}\label{lM3}
2^{-q^k (1-qs)}-2^{-q^k}\le L_k\le (1+\eps) 2^{-q^k (1-qs)}, 
\end{align}
for  $k\ge k_0$. Indeed, setting $a:=2^{1-qs}>1$ and $b:=q>1$, by Lemma~\ref{lem:reihe}, there is a $k_0$ such that for $k\ge k_0$
$$
L_k\le\sum_{i=k}^\infty 2^{q^{i+1}\cdot s}\cdot 2^{-q^{i}}=\sum_{i=k}^\infty (2^{(1-qs)})^{-q^{i}}=\sum_{i=k}^\infty a^{-b^{i}}\le (1+\eps) a^{-b^k}=(1+\eps) 2^{-q^k (1-qs)}. 
$$
The lower bound for $L_k$ follows immediately, from $N_k r_k\le L_k$ and $N_k=[2^{q^{k+1}\cdot s}]\ge 2^{q^{k+1}\cdot s}-1$.

Recall from \eqref{uS2} and \eqref{uS3} that there exist $k_0$ such that $M_k$ is bounded as follows for $k\ge k_0$:
\begin{align} \label{lM4}
2^{q^{k}\cdot s}\le \frac{M_k}2 \le (1+\eps)\cdot 2^{q^k \cdot s}. 
\end{align}
It is obvious that $k_0$ can be chosen such that both inequalities \eqref{lM3} and \eqref{lM4} hold for $k\ge k_0$.
We infer that 
\begin{align*}
\frac{2x_k}{r_k}&=\frac{1-t}t \frac{2L_k}{M_k r_k}\ge \frac{1-t}t \frac{(2^{-q^k (1-qs)}-2^{-q^k})}{(1+\eps)2^{q^k \cdot s}\cdot 2^{-q^k}}\\
&=\frac{1-t}{t}\frac 1{1+\eps} \left(2^{q^k\cdot s(q-1)}-2^{-q^k\cdot s}\right)\to\infty \text{ as } k\to\infty,
\end{align*}
since $q>1$. Hence $r_k<2x_k$ for $k$ sufficiently large.
Similarly, we obtain
\begin{align*}
\frac{2x_k}{r_{k-1}}&=\frac{1-t}t \frac{2L_k}{M_k r_{k-1}}\le \frac{1-t}t \frac{(1+\eps)\cdot 2^{-q^k (1-qs)}}{2^{q^k \cdot s}\cdot 2^{-q^{k-1}}}\\
&=
\frac{1-t}{t}(1+\eps) 2^{-q^{k-1}\cdot (q(1-qs)+qs-1)}\to 0 \text{ as } k\to\infty,
\end{align*}
since $(q-1)(1-qs)>0$. Hence $2x_k<r_{k-1}$ for $k$ sufficiently large. This completes the proof of \eqref{eq:xk-in}.

The inequalities in \eqref{eq:xk-in} imply that the lower $t$-dimensional Minkowski content of $F$ is given by
$$
\kappa_{1-t}\underline{\sM}^t(F)=\liminf_{r\to 0} r^{t-1} \lambda_1(F_r) =\liminf_{k\to\infty} h_{M_k,L_k, t}(x_k).
$$
By Lemma~\ref{lem:minimum}, we have
$$
h_{M_k,L_k, t}(x_k)=t^{-t} (1-t)^{t-1} L_k^t M_k^{1-t}.
$$
Therefore, it remains to compute
\begin{align}\label{lM2}
X_t:=\liminf_{k\to\infty} L_k^t\cdot M_k^{1-t}.
\end{align}

Using again \eqref{lM3} and \eqref{lM4}, we infer that  
on the one hand
\begin{align*}
X_t&\le \liminf_{k\to\infty} (1+\eps)^t (2^{-q^k (1-qs)})^t \cdot(1+\eps)^{1-t} (2^{q^k\cdot s})^{1-t}\\
&= (1+\eps) \lim_{k\to\infty} 2^{-q^{k}(t-qst-s+st)}\\
&=(1+\eps) \lim_{k\to\infty} 2^{-q^{k}\cdot s(\frac tm-1)},  
\end{align*}
for each $k\ge k_0$, where we took into account that $sq=1+s-\frac sm$.
On the other hand, 
\begin{align*}
X_t&\ge \liminf_{k\to\infty}  \left(2^{-q^k (1-qs)}-2^{-q^k}\right)^t  (2^{q^k\cdot s})^{1-t}\\
&=  \lim_{k\to\infty} \left((2^{-q^k (1-qs)}-2^{-q^k})\cdot 2^{q^k\cdot \frac st(1-t)} \right)^t \\
&= \lim_{k\to\infty} \left(2^{-q^k\cdot s(\frac 1 m-\frac 1 t)}-2^{-q^k(1+s-\frac st)} \right)^t 
\end{align*}  
Since the above estimates hold for each $\eps>0$, we conclude for the choice $t=m$ that $X_m=1$ and thus
$$
\kappa_{1-m}\underline{\sM}^m(F)=m^{-m} (1-m)^{m-1}.
$$
Hence $\underline{\sM}^m(F)$ is positive and finite, which implies $\ldim_M F=m$.
\end{proof}

\begin{rem}
It has has been pointed out by the referee that the function $x\mapsto h_{M_k,L_k, t}(x)$ used in the proof above is essentially equal to the function $\eps\mapsto L_D(\eps,j)$ (with $j=k$) used in the proof of \cite[Theorem 4.1, cf.\ the first equation on p.41]{LapPo1}. This is natural since in both cases Minkowski contents are computed. However, the arguments given in \cite{LapPo1} do not apply to the situation here. While for the sets considered in \cite[Theorem 4.1]{LapPo1} (or, more precisely, for the corresponding fractal strings) the Minkowski content exists, this is no longer true for the sets $F(s,m)$ sudied here. 
Nevertheless, it might be interesting to study more deeply the connections between the arguments in both cases. 
\end{rem}

\section{Proof of Theorem~\ref{thm:Fd}}\label{proof2}

We will first discuss a number of statements regarding the upper and lower dimensions of product sets. The assertions of Theorem~\ref{thm:Fd} will be an easy consequence. Before we start with the Minkowski dimensions we recall some useful alternative definitions of Minkowski and S-dimension and 
clarify  some notational problem regarding parallel sets in Remark \ref{rem:par}.

It is well known and easily verified, that if the Minkowski dimension of a compact set $A\subset\R^d$ exists, it is equivalently given by
\begin{equation}\label{eq:altdimM}
\dim_M A=d+\lim_{r\to 0}\frac{\log \lambda_d(A_r)}{-\log r}.
\end{equation}
Similarly, lower and upper Minkowski dimension are given by the same expression with the $\lim$ replaced by $\liminf$ and $\limsup$, respectively, see for instance \cite[Proposition 5.1]{fal1}.
In the same way, lower and upper S-dimension can be defined using a log-log ratio.
The lower S-dimension of a compact set $A\subset\R^d$ is given by
\begin{equation}\label{eq:altdimS}
\ldim_S A=d-1+\liminf_{r\to 0}\frac{\log \Ha^{d-1}(\partial A_r)}{-\log r}
\end{equation} 
and $\udim_S A$ by the same expression with $\liminf$ replaced by $\limsup$.
Finally, we recall the definition of the box counting dimension $\dim_B$, which is well known to coincide with the Minkowski dimension. For $r>0$, let $N_r(A)$ denote the minimum number of boxes of side length $r$ needed to cover a set $A\subset\R^d$. Then
$$
\ldim_B A:=\liminf_{r\to 0}\frac{\log N_r(A)}{-\log r} \quad \mbox{ and } \quad \udim_B A:=\limsup_{r\to 0}\frac{\log N_r(A)}{-\log r}.
$$
Below we will switch between the different definitions of the dimensions and use whatever is most convenient. 
\begin{rem} \label{rem:par}
The notion of parallel set of a set $A$ depends on the ambient space in which $A$ is considered and the notation $A_r$ does not take care of this. For instance, for an interval $I$ in $\R^2$, i.e., the convex hull of two points in $\R^2$, the $r$-parallel set with respect to the affine hull of $I$ is still an interval while the $r$-parallel set with respect to $\R^2$ is a two-dimensional set. Usually it is clear from the context what the ambient space is. However, for product sets $A\times B$, $A\subseteq\R^n$, $B\subseteq \R^m$ as occuring in the proofs below, the notation $A_r$ may cause irritations, since $A$ may be viewed as a subset of $\R^n$ but also naturally as a subset of $\R^n\times\R^m$.  To avoid any confusion, we will use the convention to denote by $A_r$ the parallel set in  $\R^n$ and by $(A\times\{0\})_r$ the parallel set in $\R^n\times\R^m$.  
\end{rem}

\begin{lem} \label{lem:Cartes-udim}
Let $A\subset\R^n$ and $B\subset\R^m$ be compact sets. Then
\begin{enumerate}
\item[(i)] $\udim_M(A\times B)\le \udim_M A+\udim_M B$\,,
\item[(ii)] $\ldim_M(A\times B)\le \ldim_M A+\udim_M B$\,.
\end{enumerate}
\end{lem}
\begin{proof}
(i) is well known, cf.~for instance \cite[Lemma 7.3]{fal1}. (ii) follows by a similar argument: Recall that $N_r(C)$ denotes the minimum number of boxes of side length $r$ needed to cover a set $C\subset\R^d$. Observe that 
$$
N_r(A\times B)\le N_r(A)\cdot N_r(B).
$$
Hence 
\begin{align*}
\ldim_M(A\times B)&= \liminf_{r\to 0}\frac{\log N_r(A\times B)}{-\log r}\le \liminf_{r\to 0}\frac{\log N_r(A) + \log N_r(B)}{-\log r}\\
&\liminf_{r\to 0}\frac{\log N_r(A)}{-\log r}+\limsup_{r\to 0}\frac{\log N_r(B)}{-\log r}=\ldim_M A +\udim_M B\,,
\end{align*}
as asserted.
\end{proof}

\begin{prop} \label{prop:Cartes-Mdim}
Let $A\subset \R^n$ and $B\subset\R^{m}$ be compact sets with $\lambda_{m}(B)>0$. Then
\begin{enumerate}
\item[(i)] $\udim_M(A\times B)=\udim_M A + m$\,,
\item[(ii)] $\ldim_M(A\times B)= \ldim_M A+ m$\,.
\end{enumerate}
\end{prop}
\begin{proof}
Note that $\dim_M B = m$. Hence the ``$\le$''-relation in (i) and (ii) follows immediately from Lemma~\ref{lem:Cartes-udim}. For the reversed inequalities recall formula \eqref{eq:altdimM} from above.
Observe that
$$
\lambda_n(A_r)\cdot\lambda_{m}(B)\le \lambda_{n+m}((A\times B)_r)
$$
which follows from the set inclusion 
$$
A_r\times B\subseteq (A\times B)_r
$$ 
and Fubini. Hence, for $0<r<1$,
$$
\frac{\log \lambda_{n+m}((A\times B)_r)}{- \log r}\ge \frac{\log \lambda_n(A_r)+\log \lambda_{m}(B)}{- \log r}.
$$
Taking the limes superior as $r\to 0$, we get
\begin{align*}
\udim_M(A\times B)&=(n+m)+\limsup_{r\to 0}\frac{\log \lambda_{n+m}((A\times B)_r)}{- \log r}\\
&\ge m+ n+\limsup_{r\to 0}\frac{\log \lambda_n(A_r)}{- \log r}=m + \udim_M A,
\end{align*}
proving (i).
The inequality $\ldim_M(A\times B)\ge \ldim_M A+ m$ follows analogously by taking the limes inferior.
\end{proof}
Now we turn our attention to the S-dimensions.
Note that assertion (i) of Lemma~\ref{lem:Cartes-udim} holds similarly with $\udim_M$ replaced by $\udim_S$ provided $\lambda_n(A)=\lambda_m(B)=0$, since both dimensions coincide in this case, see Theorem~\ref{thmU}. Unfortunately, this is not useful in the situation of Theorem~\ref{thm:Fd}, since the set $\dbox$ occuring in $F_d=F\times \dbox$ has Lebesgue measure $1$. However, for the equivalence $\udim_S(A\times B)=\udim_M(A\times B)$ it is  sufficient that one of the sets $A,B$ has zero Lebesgue measure, since this implies Lebesgue measure zero for the product set.
Clearly, the counterpart of Proposition~\ref{prop:Cartes-Mdim}(i) for $\udim_S$ is also valid under this additional hypothesis. 
\begin{cor} \label{cor:Cartes-uSdim}
Let $A\subset \R^n$ and $B\subset\R^{m}$ be compact sets with $\lambda_{n}(A)=0$. Then $\udim_S(A\times B)=\udim_M(A\times B)$. If, additionally, $\lambda_{m}(B)>0$ then $\udim_S(A\times B)=\udim_S A + m$.
\end{cor}
The situation for the lower S-dimension is more delicate.
Curiously and in contrast to the situation for the other three dimensions considered, for the lower S-dimension, the lower bound is easier to establish than the upper bound.

\begin{prop} \label{prop:Cartes-lSdim}
Let $d\ge2$ and  let $F\subset \R$ and $B\subset\R^{d-1}$ be compact sets with $\lambda_{d-1}(B)>0$.
Then
$$\ldim_S(F\times B)\ge \ldim_S F+d-1.$$ 
\end{prop}
\begin{proof} Recall \eqref{eq:altdimS}.
For each of the finitely many points $x\in\partial F_r$ we have $\{x\}\times B\subset \partial (F\times B)_r$. Since $\Ha^{d-1}(\{x\}\times B)=\lambda_{d-1}(B)$, we get  
$$
\Ha^{0}(\partial F_r)\lambda_{d-1}(B)\le \Ha^{d-1}(\partial (F\times B)_r) .
$$
Hence 
$$
\frac{\log \Ha^{d-1}(\partial (F\times B)_r)}{- \log r}\ge \frac{\log \Ha^{0}(\partial F_r)}{- \log r}+\frac{\log \lambda_{d-1}(B)}{- \log r},
$$
for $0<r<1$. Taking the limes inferior as $r\to 0$ (and noting that second term on the right hand side vanishes), we obtain
\begin{align*}
\ldim_S(F\times B)&=d-1+\liminf_{r\to 0}\frac{\log \Ha^{d-1}(\partial (F\times B)_r)}{- \log r}\\
&\ge d-1 + \liminf_{r\to 0}\frac{\log \Ha^{0}(\partial F_r)}{- \log r}=d-1 + \ldim_S F,
\end{align*}
as claimed. 
\end{proof}

We will now show that the reversed inequality in Proposition~\ref{prop:Cartes-lSdim} does also hold at least in the special case $B=\dbox$.

\begin{prop} \label{prop:Cartes-lSdim-upper}
Let $F\subset\R$ be compact. Then
$$\ldim_S(F\times \dbox)=\ldim_S F+d-1.$$ 
\end{prop}

\begin{proof}
In view of Proposition~\ref{prop:Cartes-lSdim}, it remains to prove the `$\le$'-relation. 
Let $\sL=(l_j)_{j=1}^\infty$ be the fractal string associated to $F$ encoding the lengths of the bounded complementary intervals $I_j$ of $F$. Clearly, we have $L:=\sum_{j=1}^\infty l_j =\lambda_1(I\setminus F)<\infty$, where $I$ is the convex hull of $F$. Recall that by definition of $\sL$, $l_1\ge l_2\ge l_3\ge\ldots\ge 0$. We can assume that there are infinitely many $l_j$'s different from zero. Otherwise $F$ is a finite union of intervals and singletons and the statement is obvious. 

To illustrate the idea, we will first discuss the case $d=2$. The proof in higher dimensions is similar and will be addressed afterwards.
First observe that the boundary length of $(F\times [0,1])_r\subset\R^2$ does only depend on $\sL$ (and on $\lambda_1(F)$) but not on the set $F$ itself. Indeed, this is easily seen by slicing $\R^2$ in the direction of the second coordinate and computing the measure of $\partial (F\times [0,1])_r$ in each slice separately. We have the disjoint union
$$
\R^2=(F\times\R)\cup (\R\setminus I\times\R) \cup \bigcup_{j=1}^\infty (I_j\times \R).
$$ 
In the slices of this decomposition we have, for each $r>0$, 
\begin{align}
\Ha^1\left(\partial (F\times [0,1])_r \cap(F\times\R)\right)&=2\lambda_1(F)\label{eqn:Ha1}\\
\Ha^1\left(\partial (F\times [0,1])_r \cap (I^c\times\R)\right)&=2+2\pi r\label{eqn:Ha2}
\end{align}
and 
\begin{align}\label{eqn:Ha3}
\Ha^1\left(\partial (F\times [0,1])_r \cap (I_j\times\R)\right)
&=\left\{
\begin{array}{ll} 
2+2\pi r& \text{ if } l_j>2r\\ 
4r\arcsin(\frac{l_j}{2r}) & \text{ if } l_j\le 2r
\end{array} 
\right.
\end{align}
Since $\arcsin(x)\le \frac\pi 2 x$ for $x\in[0,1]$, the last expression is bounded from above by $\pi l_j$.
Hence, writing $\tilde F:=F\times [0,1]$, we get
\begin{align*}
\Ha^1(\bd \tilde F_r)&= \Ha^1(\bd \tilde F_r\cap(F\times\R))+\Ha^1(\bd \tilde F_r\cap(I^c\times\R))\\
&\qquad + \sum_{j:l_j>2r}\Ha^1(\bd \tilde F_r\cap(I_j\times\R)) +\sum_{j: l_j\le 2r}\Ha^1(\bd \tilde F_r\cap(I_j\times\R))\\
&\le 2\lambda_1(F)+2(1+\pi r)+\sum_{j: l_j>2r} 2(1+\pi r) +\pi \sum_{j: l_j\le 2r} l_j.
\end{align*}
Now observe that $\Ha^0(\bd F_r)=2+2\cdot\#\{j:l_j>2r\}$ and that $\Ha^0(\bd F_r)\to\infty$ as $r\to 0$, which is due to the assumption that infinitely many $l_j$'s are non-zero. Moreover, the last sum is bounded from above by $\pi L$. 
Hence
\begin{align*}
\Ha^1(\bd \tilde F_r)&\le \Ha^0(\bd F_r) (1+\pi r) +2\lambda_1(F)+\pi L\\
&\le 3 \Ha^0(\bd F_r),
\end{align*}
provided $r$ is sufficiently small (namely such that $\pi r\le 1$ and $\Ha^0(\bd F_r)\ge2\lambda_1(F)+\pi L$).  
Taking logarithms and dividing by $-\log r$, we get 
\begin{align*}
\frac{\log \Ha^1(\bd \tilde F_r)}{-\log r} &\le \frac{\log \Ha^0(\bd F_r)+\log 3}{-\log r}.
\end{align*}
Thus
\begin{align*}
\ldim_S(F\times [0,1])&=1+\liminf_{r\to 0}\frac{\log \Ha^{d-1}(\partial (F\times [0,1])_r)}{- \log r}\\
&\le 1 + \liminf_{r\to 0}\frac{\log \Ha^{0}(\partial F_r)}{- \log r}=1 + \ldim_S F,
\end{align*}
which completes the proof for the case $d=2$. 

For $d>2$, the formulas \eqref{eqn:Ha1} -- \eqref{eqn:Ha3} are different, but the arguments are essentially the same.
Setting $\tilde F:=F\times\dbox$, for $r>0$, we have 
\begin{align}
\Ha^{d-1}\left(\partial \tilde F_r \cap(F\times\R^{d-1})\right)&=\Ha^{d-2}(\partial (\dbox)_r)\cdot \lambda_1(F)\,,\tag{4.1'}\label{eqn:Ha1'}\\
\Ha^{d-1}\left(\partial \tilde F_r \cap (I^c\times\R)\right)&=\Ha^{d-1}\left(\partial (\{0\}\times \dbox)_r\right)\,,\tag{4.2'}\label{eqn:Ha2'}\\
\Ha^{d-1}\left(\partial \tilde F_r \cap (I_j\times\R)\right)
&=\Ha^{d-1}\left(\partial (\{0\}\times \dbox)_r\right)\,, & \text{ if } l_j>2r\,, \tag{4.3'} \label{eqn:Ha3'}\\
\intertext{and} 
\Ha^{d-1}\left(\partial \tilde F_r \cap (I_j\times\R)\right)
&\le \Ha^{d-2}(\partial (\dbox)_r) \pi l_j\,,  &\text{ if } l_j\le 2r\,. \tag{4.3''} \label{eqn:Ha3"}
\end{align}
It is now important to note that all these expressions are bounded from above by constants which depend on $d$ (and $F$) but not on $r\in(0,1]$. More precisely, \eqref{eqn:Ha1'} is bounded by some constant $c_1=c_1(d,F)$, \eqref{eqn:Ha2'} and \eqref{eqn:Ha3'} by some constant $c_2=c_2(d)$ and \eqref{eqn:Ha3"} by $c_3\cdot l_j$ for some constant $c_3=c_3(d)$. Hence
\begin{align*}
\Ha^{d-1}\left(\partial \tilde F_r\right)&\le c_1+  c_2 +\sum_{j: l_j>2r} c_2 +c_3 \sum_{j: l_j\le 2r} l_j\\
&\le \frac{c_2}2 \Ha^{0}(\partial F_r)+ c_1+c_3 L\\
&\le \left(\frac{c_2}2+1\right) \Ha^{0}(\partial F_r)\,,
\end{align*}
provided $r$ is sufficiently small. From this inequality, the assertion for $d\ge 3$ follows as in the case $d=2$ above. 
\end{proof}

\begin{proof}[Proof of Theorem~\ref{thm:Fd}]
Combining Proposition~\ref{prop:Cartes-Mdim} and Theorem~\ref{thm:F}, we conclude that the set $F_d = F_d(s,m)\subset\R^d$ has 
$\udim_M F_d = q\cdot s +d-1$ and $\ldim_M F_d=m+d-1$. Since $\lambda_1(F)=0$, Corollary~\ref{cor:Cartes-uSdim} implies immediatly that also $\udim_S F_d = q\cdot s +d-1$. Finally, from Proposition~\ref{prop:Cartes-lSdim-upper}, we get $\ldim_S F_d = s+d-1$, which completes the proof.   
\end{proof}

\end{document}